\documentclass[11pt,leqno]{amsart}

\usepackage[a4paper, right=20mm, left=20mm, top=25mm]{geometry}

\usepackage{microtype}

\usepackage{amsthm}
\usepackage[english]{babel}
\usepackage[utf8]{inputenc}
\usepackage{csquotes}
\usepackage{hyperref}
\usepackage[foot]{amsaddr}

\usepackage{amsmath}
\usepackage{amssymb}
\usepackage{mathtools}
\usepackage{enumitem}

\usepackage{tikz}
\usetikzlibrary{matrix,arrows,positioning}

\numberwithin{equation}{section}
\numberwithin{figure}{section}

\theoremstyle{plain}
\newtheorem{Theorem}{Theorem}[section]
\newtheorem*{Theorem*}{Theorem}
\newtheorem{Lemma}[Theorem]{Lemma}
\newtheorem*{Lemma*}{Lemma}
\newtheorem{Proposition}[Theorem]{Proposition}
\newtheorem*{Proposition*}{Proposition}
\newtheorem{Corollary}[Theorem]{Corollary}

\theoremstyle{definition}
\newtheorem{Definition}[Theorem]{Definition}
\newtheorem*{Definition*}{Definition}
\newtheorem{Remark}[Theorem]{Remark}

\newcommand{\C}{\mathbb{C}}
\newcommand{\R}{\mathbb{R}}

\newcommand{\Z}{\mathbb{Z}}
\newcommand{\kk}{\Bbbk}

\newcommand{\id}{\mathrm{id}}

\newcommand{\Cdot}{\boldsymbol{\cdot}}

\newcommand{\Rep}{\mathrm{Rep}}
\newcommand{\Tilt}{\mathrm{Tilt}}

\newcommand{\im}{\mathrm{im}}

\providecommand{\abs}[1]{\lvert#1\rvert}

\title[Tensor ideals for quantum groups]{Tensor ideals for quantum groups\\via minimal tilting complexes}
\author{Jonathan Gruber}
\address{Department of Mathematics, National University of Singapore, Singapore}
\email{jgruber@nus.edu.sg}
\subjclass{20G42 (primary), 17B55, 18M15, 20G05 (secondary)}
\keywords{quantum group, tensor ideal, tilting module}
\date{\today}

\begin{document}

\begin{abstract}
	We use minimal tilting complexes to construct an explicit bijection between the set of thick tensor ideals with the two-out-of-three property in the category of finite-dimensional modules over a quantum group at a root of unity and the set of thick tensor ideals in the subcategory of tilting modules.
	We also explain why the analogous construction for rational representations of a reductive algebraic group over a field of positive characteristic does not give rise to a bijection.
\end{abstract}

\maketitle

\section*{Introduction}

One way of understanding the structure of a braided monoidal category $\mathcal{C}$ is via its thick tensor ideals, i.e.\ the sets of objects of $\mathcal{C}$ that are closed under direct sums, retracts and tensor products with arbitrary objects of $\mathcal{C}$.
When $\mathcal{C}$ is triangulated and the tensor product bifunctor is exact in both arguments then one can further impose that tensor ideals should be well-behaved with respect to the triangulated structure (in a suitable sense), and this leads to the rich theory of tensor triangular geometry developed by P.\ Balmer \cite{Balmerspectrum}.
When $\mathcal{C}$ is abelian, and the tensor product bifunctor is exact in both arguments, we can consider the set of thick tensor ideals $\mathcal{J}$ with the so-called \emph{two-out-of-three property}, that is, with the property that for any short exact sequence $0 \to A \to B \to C \to 0$ in $\mathcal{C}$ such that two of the three objects $A$, $B$ and $C$ belong to $\mathcal{J}$, the third one also belongs to $\mathcal{J}$.
%One motivation for studying thick tensor ideals with the two-out-of-three property is that every additive monoidal functor $F \colon \mathcal{C} \to \mathcal{T}$ from the abelian category $\mathcal{C}$ to a triangulated category $\mathcal{T}$ that sends short exact sequences to distinguished triangles gives rise to such a tensor ideal via $\mathcal{J} = \{ M \in \mathrm{Ob}(\mathcal{C}) \mid F(M) = 0 \}$.

In this note, we study thick tensor ideals with the two-out-of-three property in the category of modules over a quantum group at a root of unity.
Let $U = U_\zeta( \mathfrak{g} )$ be the quantum group at a primitive $\ell$-th root of unity $\zeta \in \C$ corresponding to a complex simple Lie algebra $\mathfrak{g}$, defined using divided powers as in \cite[Appendix H]{Jantzen}, and assume that $\ell$ is odd and strictly greater than the Coxeter number $h$ of $\mathfrak{g}$, and not divisible by $3$ if $\mathfrak{g}$ is of type $\mathrm{G}_2$.
Further, let $\Rep(U)$ be the category of finite-dimensional $U$-modules of type one and let $\Tilt(U)$ be the full subcategory of tilting $U$-modules, again as in \cite[Appendix H]{Jantzen}.
The thick tensor ideals in $\Tilt(U)$ have been classified by V.\ Ostrik in \cite{OstrikTensorIdeals} in terms of certain Kazhdan-Lusztig cells in the affine Weyl group of $U$.
Here, we provide an explicit bijection between the set of thick tensor ideals in $\Tilt(U)$ and the set of thick tensor ideals with the two-out-of-three property in $\Rep(U)$, using the theory of \emph{minimal tilting complexes}.
As explained in \cite[Subsection 2.2]{GruberMinimalTilting}, we can associate to every $U$-module $M$ a bounded minimal complex $C_\mathrm{min}(M)$ in $\Tilt(U)$ (unique up to isomorphism),
%whose cohomology groups are given by
%\[ H^i\big( C_\mathrm{min}(M) \big) \cong \begin{cases} M & \text{if } i=0 , \\ 0 & \text{otherwise} \end{cases} , \]
and we call $C_\mathrm{min}(M)$ the minimal tilting complex of $M$.
Given a thick tensor ideal $\mathcal{I}$ in $\Tilt(U)$, we define
\[ \langle \mathcal{I} \rangle \coloneqq \big\{ M \in \Rep(U) \mathop{\big|} \text{all terms of } C_\mathrm{min}(M) \text{ belong to } \mathcal{I} \big\} \]
and prove that $\langle \mathcal{I} \rangle$ is a thick tensor ideal in $\Rep(U)$ with the two-out-of-three property.
Our main result is as follows; see Theorem \ref{thm:bijection}.

\begin{Theorem*}
	The map $\mathcal{I} \mapsto \langle \mathcal{I} \rangle$ from the set of thick tensor ideals in $\Tilt(U)$ to the set of thick tensor ideals in $\Rep(U)$ with the two-out-of-three property is a bijection, with inverse map $\mathcal{J} \mapsto \mathcal{J} \cap \Tilt(U)$.
\end{Theorem*}

Thick tensor ideals in $\Rep(U)$ have also been studied by B.\ Boe, J.\ Kujawa and D.\ Nakano in \cite{BoeKujawaNakano} via the geometry of the nilpotent cone of $\mathfrak{g}$, and we partially rely on their results.

Our construction also makes perfect sense for the category $\Rep(G)$ of finite-dimensional rational representations of a simply-connected simple algebraic group $G$ over an algebraically closed field of positive characteristic $p \geq h$, but the above theorem fails in this setting.
In Section \ref{sec:modular}, we show that the Frobenius twist of the Steinberg module for $G$ generates a thick tensor ideal in $\Rep(G)$ with the two-out-of-three property which is not of the form $\langle \mathcal{I} \rangle$, for any thick tensor ideal $\mathcal{I}$ in the category $\Tilt(G)$ of tilting modules in $\Rep(G)$.

\subsection*{Acknowledgements}

The author would like to thank Stephen Donkin, Thorge Jensen, Daniel Nakano and Donna Testerman for helpful conversations and comments.
This work was funded by the Swiss National Science Foundation under the grants FNS 200020\_175571 and FNS 200020\_207730 and by the Singapore MOE grant R-146-000-294-133. 

\section{Minimal tilting complexes} \label{sec:introduction}

Let us start by recalling some important facts about representations of quantum groups and minimal tilting complexes.
We keep the notation from the introduction; in particular $U = U_\zeta(\mathfrak{g})$ is a quantum group at a primitive $\ell$-th root of unity $\zeta \in \C$  as in \cite[Appendix H]{Jantzen}, where $\ell$ is odd (and not divisible by $3$ if $\mathfrak{g}$ is of type $\mathrm{G}_2$).
We write $\Rep(U)$ for the category of finite-dimensional $U$-modules of type one.
Note that $\Rep(U)$ is a rigid monoidal category because $U$ is a Hopf algebra and that $\Rep(U)$ admits a braiding by \cite[Chapter 32]{LusztigQuantumGroups}.
In the following, we simply refer to the objects of $\Rep(U)$ as $U$-modules.
The category $\Rep(U)$ is also a highest weight category with weight poset $(X^+,\leq)$, where $X^+$ denotes the set of dominant weights for $\mathfrak{g}$ (with respect to a fixed base of the root system of $\mathfrak{g}$) and $\leq$ is the usual dominance order.
We write $\Tilt(U)$ for the full subcategory of tilting modules in $\Rep(U)$; see Section 1 and Subsection 4.3 in \cite{GruberMinimalTilting} for more details.
As explained in \cite[Subsection 2.2]{GruberMinimalTilting}, we can assign to every $U$-module $M$ a \emph{minimal tilting complex} $C_\mathrm{min}(M)$; it is the unique minimal bounded complex in $\Tilt(U)$ whose cohomology groups are given by
\[ H^i\big( C_\mathrm{min}(M) \big) \cong \begin{cases} M & \text{if } i=0 , \\ 0 & \text{otherwise} \end{cases} \]
for $i \in \Z$.
(See \cite[Definition 2.2]{GruberMinimalTilting} for the definition of a minimal complex.)
Below, we list some key properties of minimal tilting complexes.

\begin{Lemma} \label{lem:directsum}
	Let $M$ and $N$ be $U$-modules.
	Then we have $C_\mathrm{min}(M \oplus N) \cong C_\mathrm{min}(M) \oplus C_\mathrm{min}(N)$.
\end{Lemma}
\begin{proof}
	See part (1) of Lemma 2.12 in \cite{GruberMinimalTilting}.
\end{proof}

For $U$-modules $M$ and $N$, let us write $M \stackrel{\oplus}{\subseteq} N$ if $M$ admits a split embedding into $N$, and denote by $C_\mathrm{min}(M)_i$ the term in homological degree $i \in \Z$ of $C_\mathrm{min}(M)$.

\begin{Lemma} \label{lem:shortexactsequence}
	Let $0 \to A \to B \to C \to 0$ be a short exact sequence of $U$-modules. Then
	% modification here
	\begin{align*}
	C_\mathrm{min}(A)_i & \stackrel{\oplus}{\subseteq} C_\mathrm{min}(B)_i \oplus C_\mathrm{min}(C)_{i-1} , \\
	C_\mathrm{min}(B)_i & \stackrel{\oplus}{\subseteq} C_\mathrm{min}(A)_i \oplus C_\mathrm{min}(C)_i , \\
	C_\mathrm{min}(C)_i & \stackrel{\oplus}{\subseteq} C_\mathrm{min}(A)_{i+1} \oplus C_\mathrm{min}(B)_i
	\end{align*}
	for all $i\in\Z$.
\end{Lemma}
\begin{proof}
	This follows from Lemma 2.17 in \cite{GruberMinimalTilting}.
\end{proof}

Given two complexes $X = ( \cdots \to X_i \xrightarrow{d_i} X_{i+1} \to \cdots )$ and $Y = ( \cdots \to Y_i \xrightarrow{d_i^\prime} Y_{i+1} \to \cdots )$ of $U$-modules, we define the tensor product complex $X \otimes Y$ with terms
\[ (X \otimes Y)_i = \bigoplus_{j+k=i} X_j \otimes Y_k \]
for $i \in \Z$ and differential defined on $X_j \otimes Y_k$ via $d_j \otimes \id_{Y_k} + (-1)^j \cdot \id_{X_j} \otimes d^\prime_k$.
The cohomology groups of $X \otimes Y$ can be computed via the Künneth-formula (see for instance \cite[Theorem 4.1]{KunnethFormula}):
\[ H^i(X \otimes Y) = \bigoplus_{j+k=i} H^j(X) \otimes H^k(Y) \]
%as the total complex of the double complex in Figure \ref{fig:doublecomplex}. 
%\begin{figure} \label{fig:doublecomplex}
%\begin{center}
%	\begin{tikzpicture}
%	\node (A2) at (0,1) {$\vdots$};
%	\node (A3) at (4,1) {$\vdots$};
%	\node (B1) at (-2,0) {$\cdots$};
%	\node (B2) at (0,0) {$X_i \otimes Y_i$};
%	\node (B3) at (4,0) {$X_{i+1} \otimes Y_i$};
%	\node (B4) at (6,0) {$\cdots$};
%	\node (C1) at (-2,-2) {$\cdots$};
%	\node (C2) at (0,-2) {$X_i \otimes Y_{i+1}$};
%	\node (C3) at (4,-2) {$X_{i+1} \otimes Y_{i+1}$};
%	\node (C4) at (6,-2) {$\cdots$};
%	\node (D2) at (0,-3) {$\vdots$};
%	\node (D3) at (4,-3) {$\vdots$};
%
%	\draw[->] (B1) -- (B2);
%	\draw[->] (B2) -- node[above]{\footnotesize $d_i \otimes \id_{Y_i}$} (B3);
%	\draw[->] (B3) -- (B4);
%	\draw[->] (C1) -- (C2);
%	\draw[->] (C2) -- node[below]{\footnotesize $d_i \otimes \id_{Y_{i+1}}$} (C3);
%	\draw[->] (C3) -- (C4);
%	\draw[->] (A2) -- (B2);
%	\draw[->] (B2) -- node[left]{\footnotesize $\id_{X_i} \otimes d_i^\prime$} (C2);
%	\draw[->] (C2) -- (D2);
%	\draw[->] (A3) -- (B3);
%	\draw[->] (B3) -- node[right]{\footnotesize $\id_{X_{i+1}} \otimes d_i^\prime$} (C3);
%	\draw[->] (C3) -- (D3);
%	\end{tikzpicture}
%\end{center}
%\caption{The double complex corresponding to $X$ and $Y$}
%\end{figure}

\begin{Lemma} \label{lem:tensorproduct}
	Let $M$ and $N$ be $U$-modules.
	Then there is a split monomorphism
	\[ C_\mathrm{min}(M \otimes N) \longrightarrow C_\mathrm{min}(M) \otimes C_\mathrm{min}(N) \]
	in the category of complexes in $\Tilt(U)$.
\end{Lemma}
\begin{proof}
	By the Künneth formula, the tensor product complex satisfies
	\[ H^i\big( C_\mathrm{min}(M) \otimes C_\mathrm{min}(N) \big) \cong \begin{cases} M \otimes N & \text{if } i=0 , \\ 0 & \text{otherwise} \end{cases} \]
	for $i \in \Z$, whence we have $C_\mathrm{min}(M) \otimes C_\mathrm{min}(N) \cong C_\mathrm{min}(M \otimes N)$ in $D^b\big( \Rep(U) \big)$, the bounded derived category of $\Rep(U)$.
	By part (2) of Lemma 2.12 in \cite{GruberMinimalTilting}, this implies that $C_\mathrm{min}(M \otimes N)$ is the minimal complex of $C_\mathrm{min}(M) \otimes C_\mathrm{min}(N)$ and that there is a split monomorphism as in the statement of the lemma.
\end{proof}

\section{Tensor ideals for quantum groups} \label{sec:tensoridealsquantum}

Recall that a thick tensor ideal in $\Rep(U)$ is a set $\mathcal{J}$ of $U$-modules that is closed under direct sums, retracts and tensor products with arbitrary $U$-modules, that is, for $U$-modules $M$ and $N$ and $X \cong M \oplus N$, we have $X \in \mathcal{J}$ if and only if $M \in \mathcal{J}$ and $N \in \mathcal{J}$, and if $M \in \mathcal{J}$ then $M \otimes N \in \mathcal{J}$.
Thick tensor ideals in $\Tilt(U)$ are defined analogously.
We say that a thick tensor ideal $\mathcal{J}$ in $\Rep(U)$ has the \emph{2/3-property} (or \emph{two-out-of-three property}) if for any short exact sequence
	$ 0 \to A \to B \to C \to 0 $
	of $U$-modules such that two of the $U$-modules $A$, $B$ and $C$ belong to $\mathcal{J}$, the third also belongs to $\mathcal{J}$.

\begin{Remark}
	In \cite{BoeKujawaNakano}, the term \emph{thick tensor ideal} in $\Rep(U)$ is used for what we call \emph{thick tensor ideal with the 2/3-property}.
\end{Remark}

\begin{Definition}
	For any thick tensor ideal $\mathcal{I}$ in $\mathrm{Tilt}(U)$, we define a set of $U$-modules by
	\[ \langle \mathcal{I} \rangle \coloneqq \big\{ M \mathop{\big|} M \text{ is a } U \text{-module and all terms of } C_\mathrm{min}(M) \text{ belong to } \mathcal{I} \big\} . \]
\end{Definition}

\begin{Lemma} \label{lem:twooutofthreeproperty}
	Let $\mathcal{I}$ be a thick tensor ideal in $\mathrm{Tilt}(U)$.
	Then $\langle \mathcal{I} \rangle$ is a thick tensor ideal in $\Rep(U)$ and $\langle \mathcal{I} \rangle$ has the 2/3-property.
\end{Lemma}
\begin{proof}
	First note that $\langle \mathcal{I} \rangle$ is closed under direct sums and retracts because
	\[ C_\mathrm{min}( M_1 \oplus M_2 ) = C_\mathrm{min}(M_1) \oplus C_\mathrm{min}(M_2) \]
	for all $U$-modules $M_1$ and $M_2$, by  Lemma \ref{lem:directsum}.
	If $M_2\in \langle \mathcal{I} \rangle$ then all terms of the tensor product complex $C_\mathrm{min}(M_1) \otimes C_\mathrm{min}(M_2)$ belong to $\mathcal{I}$ because $\mathcal{I}$ is a tensor ideal.
	As $C_\mathrm{min}(M_1 \otimes M_2)$ is a direct summand of $C_\mathrm{min}(M_1) \otimes C_\mathrm{min}(M_2)$ in the category of complexes in $\Tilt(U)$ by Lemma \ref{lem:tensorproduct} and as $\mathcal{I}$ is closed under retracts, we conclude that $M_1 \otimes M_2 \in \langle \mathcal{I} \rangle$.
	Finally, for a short exact sequence
	\[ 0 \to A \to B \to C \to 0 \]
	of $U$-modules, we have
	% modification here
	\begin{align*}
	C_\mathrm{min}(A)_i & \stackrel{\oplus}{\subseteq} C_\mathrm{min}(B)_i \oplus C_\mathrm{min}(C)_{i-1} , \\
	C_\mathrm{min}(B)_i & \stackrel{\oplus}{\subseteq} C_\mathrm{min}(A)_i \oplus C_\mathrm{min}(C)_i , \\
	C_\mathrm{min}(C)_i & \stackrel{\oplus}{\subseteq} C_\mathrm{min}(A)_{i+1} \oplus C_\mathrm{min}(B)_i
	\end{align*}
	for all $i\in\Z$, by Lemma \ref{lem:shortexactsequence}.
	As $\mathcal{I}$ is closed under retracts, we conclude that $\langle \mathcal{I} \rangle$ has the 2/3-property.
\end{proof}

The following Lemma justifies the notation $\langle \mathcal{I} \rangle$.

\begin{Lemma} \label{lem:tensoridealgenerated}
	Let $\mathcal{I}$ be a thick tensor ideal in $\mathrm{Tilt}(U)$.
	Then $\langle \mathcal{I} \rangle$ is the smallest thick tensor ideal with the 2/3-property in $\Rep(U)$ that contains $\mathcal{I}$.
\end{Lemma}
\begin{proof}
	The inclusion $\mathcal{I} \subseteq \langle \mathcal{I} \rangle$ follows from the fact that for every tilting $U$-module $M$, we have $C_\mathrm{min}(M)=M$, viewed as a complex concentrated in degree zero; see \cite[Remark 2.11]{GruberMinimalTilting}.
	Now let $\mathcal{J}$ be a thick tensor ideal with the 2/3-property in $\Rep(U)$ such that $\mathcal{I} \subseteq \mathcal{J}$, and let $M$ be a $U$-module in $\langle \mathcal{I} \rangle$.
	We claim that $M$ belongs to $\mathcal{J}$.
	Writing $C_\mathrm{min}(M)$ as
	\[ \cdots \xrightarrow{~d_{-2}~} T_{-1} \xrightarrow{~d_{-1}~} T_0 \xrightarrow{~\,d_0\,~} T_1 \xrightarrow{~\,d_1\,~} \cdots , \]
	we have $M \cong \ker(d_0) / \im(d_{-1})$, so there is a short exact sequence
	\[ 0 \longrightarrow \im(d_{-1}) \longrightarrow \ker(d_0) \longrightarrow M \longrightarrow 0 . \]
	As $\mathcal{J}$ has the 2/3-property, it suffices to show that $\im(d_{-1})$ and $\ker(d_0)$ belong to $\mathcal{J}$.
	As $C_\mathrm{min}(M)$ is exact in all degrees except zero, there are short exact sequences
	\[ 0 \longrightarrow \ker(d_i) \longrightarrow T_i \longrightarrow \ker(d_{i+1}) \longrightarrow 0 \]
	for all $i\geq 0$, where $T_i\in \mathcal{I} \subseteq \mathcal{J}$, and using the 2/3-property, we see that $\ker(d_i)$ belongs to $\mathcal{J}$ if and only if $\ker(d_{i+1})$ belongs to $\mathcal{J}$. Now $C_\mathrm{min}(M)$ is bounded, so $T_i=0$ for some $i>0$, and we conclude that $\ker(d_0)$ belongs to $\mathcal{J}$.
	Analogously, we see that $\im(d_{-1})$ belongs to $\mathcal{J}$, and the claim follows.
\end{proof}

For a thick tensor ideal $\mathcal{J}$ in $\Rep(U)$, it is straightforward to see that $\mathcal{J}\cap \Tilt(U)$ (the set of tilting modules in $\mathcal{J}$) is a thick tensor ideal in $\Tilt(U)$.
The next result shows that the map $\mathcal{J} \mapsto \mathcal{J}\cap \Tilt(U)$ is a section to the map $\mathcal{I} \mapsto \langle \mathcal{I} \rangle$ from the set of thick tensor ideals in $\Tilt(U)$ to the set of thick tensor ideals in $\Rep(U)$ with the 2/3-property.

\begin{Lemma} \label{lem:extensionoftensoridealsection}
	Let $\mathcal{I}$ be a thick tensor ideal in $\Tilt(U)$. Then $\langle \mathcal{I} \rangle \cap \Tilt(U) = \mathcal{I}$.
\end{Lemma}
\begin{proof}
	For a tilting $U$-module $M$, we have $C_\mathrm{min}(M) = M$, viewed as a complex concentrated in degree zero, and it follows that $M$ belongs to $\mathcal{I}$ if and only if all terms of $C_\mathrm{min}(M)$ belong to $\mathcal{I}$.
\end{proof}

Following \cite{Balmerspectrum} and \cite{BoeKujawaNakano}, let us call a proper thick tensor ideal $\mathcal{P}$ in $\Rep(U)$ a \emph{prime ideal} if $M \otimes N \in \mathcal{P}$ implies that either $M \in \mathcal{P}$ or $N \in \mathcal{P}$, for $U$-modules $M$ and $N$.
The following result has been proven by P.\ Balmer in the framework of tensor triangulated geometry in \cite[Lemma 4.2]{Balmerspectrum}.
The proof in our setting is almost identical; we sketch the main ideas below and refer the reader to \cite{Balmerspectrum} for more details.

\begin{Proposition} \label{prop:intersectionofprimeideals}
	Let $\mathcal{J}$ be a proper thick tensor ideal with the 2/3-property in $\Rep(U)$.
	Then we have
	\[ \mathcal{J} = \bigcap_{\mathcal{J} \subseteq \mathcal{P}} \mathcal{P} , \]
	where the intersection runs over all prime thick tensor ideals $\mathcal{P}$ in $\Rep(U)$ such that $\mathcal{P}$ has the 2/3-property and $\mathcal{J} \subseteq \mathcal{P}$.
\end{Proposition}
\begin{proof}
	It is clear that $\mathcal{J}$ is contained in $\bigcap_{\mathcal{J} \subseteq \mathcal{P}} \mathcal{P}$.
	Now let $M$ be a $U$-module that does not belong to $\mathcal{J}$ and consider the set $\mathcal{S} = \{ M^{\otimes n} \mid n \in \Z_{\geq 0} \}$.
	Since $M$ is a direct summand of $M \otimes M^* \otimes M$ (where $M^*$ denotes the dual of $M$) and $\mathcal{J}$ is a thick tensor ideal, the set $\mathcal{S}$ is disjoint from $\mathcal{J}$; cf.\ Remark 4.3 and Proposition 4.4 in \cite{Balmerspectrum}.
	As in Lemma 2.2 in \cite{Balmerspectrum}, one can use Zorn's lemma to prove that there exists a thick prime tensor ideal $\mathcal{P}$ in $\Rep(U)$ such that $\mathcal{P}$ has the 2/3-property, $\mathcal{J} \subseteq \mathcal{P}$ and $\mathcal{S} \cap \mathcal{P} = \varnothing$.
	In particular, $M$ does not belong to $\bigcap_{\mathcal{J} \subseteq \mathcal{P}} \mathcal{P}$ and the claim follows.
\end{proof}

Before we can prove our main result, we need two more preliminary lemmas.

\begin{Lemma} \label{lem:intersection}
	Let $(\mathcal{I}_j)_{j \in J}$ be a collection of thick tensor ideals in $\Tilt(U)$.
	Then we have
	\[ \bigcap_{ j \in J } \langle \mathcal{I}_j \rangle = \Big\langle \bigcap_{ j \in J } \mathcal{I}_j \Big\rangle . \]
\end{Lemma}
\begin{proof}
	For a $U$-module $M$ and for $i \in \Z$, the tilting module $C_\mathrm{min}(M)_i$ belongs to $\bigcap_{ j \in J } \mathcal{I}_j$ if and only if it belongs to $\mathcal{I}_j$ for all $j \in J$.
	Hence $M$ belongs to $\big\langle \bigcap_{ j \in J } \mathcal{I}_j \big\rangle$ if and only if $M$ belongs to $\bigcap_{ j \in J } \langle \mathcal{I}_j \rangle$.
\end{proof}

All the results so far are valid without any additional assumptions on the order $\ell$ of $\zeta$, apart from those that were made at the beginning of Section \ref{sec:introduction}.
For the rest of the section, we assume that $\ell>h$, the Coxeter number of $\mathfrak{g}$.

\begin{Lemma} \label{lem:mapsurjective}
	Let $\mathcal{J}$ be a thick tensor ideal in $\Rep(U)$ with the 2/3-property.
	Then we have
	\[ \mathcal{J} = \Big\langle \bigcap_{\mathcal{J} \subseteq \mathcal{P}} \mathcal{P} \cap \Tilt(U) \Big\rangle , \]
	where the intersection runs over all prime thick tensor ideals $\mathcal{P}$ in $\Rep(U)$ such that $\mathcal{P}$ has the 2/3-property and $\mathcal{J} \subseteq \mathcal{P}$.
\end{Lemma}
\begin{proof}
	By Theorem 8.2.1 in \cite{BoeKujawaNakano}, every prime thick tensor ideal $\mathcal{P}$ in $\Rep(U)$ with the 2/3-property is generated by an indecomposable tilting module $M$.
	Hence if $\mathcal{I}$ is the thick tensor ideal in $\Tilt(U)$ generated by $M$ then $\mathcal{P} = \langle \mathcal{I} \rangle$, and by Lemma \ref{lem:extensionoftensoridealsection}, we have $\mathcal{P} \cap \Tilt(U) = \mathcal{I}$ and therefore $\mathcal{P} = \big\langle \mathcal{P} \cap \Tilt(U) \big\rangle$.
	Using Proposition \ref{prop:intersectionofprimeideals} and Lemma \ref{lem:intersection}, it follows that
	\[ \mathcal{J} = \bigcap_{\mathcal{J} \subseteq \mathcal{P}} \mathcal{P} = \bigcap_{\mathcal{J} \subseteq \mathcal{P}} \big\langle \mathcal{P} \cap \Tilt(U) \big\rangle = \Big\langle \bigcap_{\mathcal{J} \subseteq \mathcal{P}} \mathcal{P} \cap \Tilt(U) \Big\rangle , \]
	as required.
\end{proof}

Now we are ready to prove our main result:

\begin{Theorem} \label{thm:bijection}
	The map $\mathcal{I} \mapsto \langle \mathcal{I} \rangle$ from the set of thick tensor ideals in $\Tilt(U)$ to the set of thick tensor ideals with the 2/3-property in $\Rep(U)$ is a bijection.
	The inverse map is given by $\mathcal{J} \mapsto \mathcal{J} \cap \Tilt(U)$.
\end{Theorem}
\begin{proof}
	For any thick tensor ideal $\mathcal{I}$ in $\Tilt(U)$, we have $\langle \mathcal{I} \rangle \cap \Tilt(U) = \mathcal{I}$ by Lemma \ref{lem:extensionoftensoridealsection}.
	Conversely, for a thick tensor ideal $\mathcal{J}$ in $\Rep(U)$ with the 2/3-property, we have
	\[ \mathcal{J} \cap \Tilt(U) = \Big( \bigcap_{ \mathcal{J} \subseteq \mathcal{P} } \mathcal{P} \Big) \cap \Tilt(U) = \bigcap_{ \mathcal{J} \subseteq \mathcal{P} } \mathcal{P} \cap \Tilt(U) \]
	by Proposition \ref{prop:intersectionofprimeideals}, and therefore
	\[ \big\langle \mathcal{J} \cap \Tilt(U) \big\rangle = \Big\langle \bigcap_{ \mathcal{J} \subseteq \mathcal{P} } \mathcal{P} \cap \Tilt(U) \Big\rangle = \mathcal{J} \]
	by Lemma \ref{lem:mapsurjective}.
\end{proof}

\begin{Remark}
	The existence of a bijection between the set of thick tensor ideals in $\Tilt(U)$ and the set of thick tensor ideals in $\Rep(U)$ with the 2/3-property can also be deduced by combining Corollary 7.7.2 and Theorem 8.1.1 in \cite{BoeKujawaNakano}.
	However, it is not clear from the results in \cite{BoeKujawaNakano} that this bijection can be described in terms of minimal tilting complexes, as we have done here.
\end{Remark}

\section{The modular case} \label{sec:modular}

Let $G$ be a simply connected simple algebraic group over an algebraically closed field $\kk$ of characteristic $p>0$, with the same root system as $\mathfrak{g}$.
Then the category $\Rep(G)$ of finite-dimensional rational $G$-modules is a highest weight category with weight poset $(X^+,\leq)$, and we write $L_\lambda$ and $T_\lambda$ for the simple $G$-module and the indecomposable tilting $G$-module of highest weight $\lambda \in X^+$, respectively.
As before, we write $\Tilt(G)$ for the full subcategory of tilting modules in $\Rep(G)$ and we simply refer to the objects of $\Rep(G)$ as $G$-modules.

We can mimic the construction from Section \ref{sec:tensoridealsquantum} to define a map $\mathcal{I} \mapsto \langle \mathcal{I} \rangle$ from the set of thick tensor ideals in $\Tilt(G)$ to the set of thick tensor ideals in $\Rep(G)$ with the 2/3-property, and all the results from the preceding sections up to (including) Lemma \ref{lem:intersection} carry over to this setting verbatim.
Below, we explain why Theorem \ref{thm:bijection} fails when we replace $\Rep(U)$ by $\Rep(G)$.
We first introduce some more notation.

Let $\Phi$ be the root system of $G$ and fix a base $\Pi$ of $\Phi$ corresponding to a positive system $\Phi^+$.
Further let $W$ be the Weyl group of $G$ and let $X$ and $X^+$ be the weight lattice of $G$ and the set of dominant weights with respect to $\Phi^+$.
A dominant weight is called \emph{$p$-restricted} if it belongs to
\[ X_1 = \{ \lambda \in X^+ \mid (\lambda,\alpha^\vee) < p \text{ for all } \alpha \in \Pi \} , \]
where $(- \,,-)$ is a $W$-invariant inner product on $X \otimes_\Z \R$ and $\alpha^\vee = \frac{ 2\alpha }{ (\alpha,\alpha) }$ denotes the coroot corresponding to $\alpha\in\Phi$, and we say that $\lambda \in X$ is \emph{$p$-regular} if $(\lambda+\rho,\beta^\vee)$ is not divisible by $p$ for any $\beta \in \Phi$, where $\rho$ is the half-sum of all positive roots.
For all $\lambda \in X^+$, we can uniquely write $\lambda = \lambda_0 + p \lambda_1$ with $\lambda_0 \in X_1$ and $\lambda_1 \in X^+$, and by Steinberg's tensor product theorem (see \cite[Proposition 3.16]{Jantzen}), we have $L_\lambda \cong L_{\lambda_0} \otimes L_{\lambda_1}^{[1]}$, where $M^{[1]}$ denotes the pullback of a $G$-module $M$ along the Frobenius morphism $\mathrm{Fr} \colon G \to G$.
For $r>0$, let us further denote by $G_r$ the \emph{$r$-th Frobenius kernel} of $G$ (i.e.\ the scheme theoretic kernel of the $r$-th power of $\mathrm{Fr}$), as in Section II.3.1 in \cite{Jantzen}.
The \emph{complexity over $G_r$} of a $G$-module $M$ is the dimension $c_{G_r}(M) = \dim V_{G_r}(M)$ of its support variety over $G_r$.
We do not recall any details of the definition here and instead refer the reader to Section 2 in \cite{NakanoParshallVella}.
Some important properties of the complexity of $G$-modules are listed below.

\begin{Lemma}
	Let $r>0$ and let $M$ and $N$ be $G$-modules.
	Then we have
	\begin{enumerate}
	\item $c_{G_r}( M \oplus N ) = \max\{ c_{G_r}(M) , c_{G_r}(N) \}$;
	\item $c_{G_r}( M \otimes N ) \leq \min\{ c_{G_r}(M) , c_{G_r}(N) \}$.
	\item For a short exact sequence $0 \to M_1 \to M_2 \to M_3 \to 0$ in $\Rep(G)$ and a permutation $\sigma$ of the set $\{1,2,3\}$, we have $c_{G_r}( M_{\sigma(1)} ) \leq \max\{ c_{G_r}( M_{\sigma(2)} ) , c_{G_r}( M_{\sigma(3)} ) \}$.
	\end{enumerate}
\end{Lemma}
\begin{proof}
	This follows from the properties of support varieties listed in (2.2.4)--(2.2.7) in \cite{NakanoParshallVella}.
\end{proof}

For $r,c>0$, the preceding lemma implies that the set
\[ \mathcal{J}_{r,\leq c} = \big\{ M \mathop{\big|} M \text{ is a } G \text{-module with } c_{G_r}(M) \leq c \big\} \]
is a thick tensor ideal in $\Rep(G)$ with the 2/3-property.
In the remainder of this section, we prove that the tensor ideal $\mathcal{J}_{2,\leq \abs{\Phi}}$ is not of the form $\langle \mathcal{I} \rangle$ for any thick tensor ideal $\mathcal{I}$ in $\Tilt(G)$ when $p \geq h$.

Suppose from now on that $p \geq h$, and note that $\mathcal{J}_{2,\leq \abs{\Phi}}$ is a proper tensor ideal in $\Rep(G)$ because the complexity of the trivial $G$-module $\kk$ over $G_2$ is the dimension $c_{G_2}(\kk) = \dim G$ of the variety of pairs of commuting nilpotent elements in the Lie algebra $\mathfrak{g}_\kk$ of $G$; see Lemma 1.7 in \cite{SuslinFriedlanderBendel1}, Theorem 5.2 in \cite{SuslinFriedlanderBendel2} and the introduction to \cite{PremetNilpotentCommuting}.

\begin{Lemma} \label{lem:SteinbergFrobeniustwist}
	The simple $G$-module $L_{(p^2-p) \cdot \rho} \cong L_{(p-1) \cdot \rho}^{[1]}$ belongs to $\mathcal{J}_{2,\leq \abs{\Phi}}$. 
\end{Lemma}
\begin{proof}
	By Theorem 2.4 in \cite{NakanoBoundComplexity}, we have
	\[ c_{G_2}( L_{(p-1) \cdot \rho}^{[1]} ) \leq c_{G_1}(\kk) + c_{G_1}( L_{(p-1)\cdot\rho} ) , \]
	where $c_{G_1}(\kk) = \abs{\Phi}$ is the dimension of the nilpotent cone of $G$ and $c_{G_1}(L_{(p-1)\Cdot\rho}) = 0$ because the Steinberg module $L_{(p-1)\Cdot\rho}$ is projective as a $G_1$-module.
	See Sections II.12.14 and II.10.2 in \cite{Jantzen}.
\end{proof}

\begin{Lemma} \label{lem:pregularhighestweight}
	Let $\lambda \in X^+$ be a $p$-regular weight.
	Then there is no proper thick tensor ideal $\mathcal{I}$ in $\Tilt(G)$ such that $L_\lambda$ belongs to $\langle \mathcal{I} \rangle$.
\end{Lemma}
\begin{proof}
	By Proposition 12 and the remarks after Proposition 9 in \cite{AndersenCells}, the category $\Tilt(G)$ has a unique maximal thick tensor ideal, whose objects are the direct sums of tilting modules $T_\lambda$ with $(\lambda+\rho,\alpha_0^\vee) \geq p$, where $\alpha_0^\vee$ is the highest coroot.
	(It is called the ideal of \emph{negligible} tilting modules, cf.\ \cite{EtingofOstrikSemisimplification}.)
	Therefore, it suffices to show that there exist $i \in \Z$ and $\mu \in X^+$ with $(\mu+\rho,\alpha_0^\vee)<p$ such that $T_\mu$ is a direct summand of $C_\mathrm{min}(L_\lambda)_i$.
	
	Let us write $d(\lambda)$ for the number of hyperplanes of the form $H_{r,\beta} = \{ x \in X \otimes_\Z \R \mid (x+\rho,\beta^\vee) = rp \}$ that separate $\lambda$ and $0$, for $\beta\in\Phi^+$ and $r>0$.
	Then $d(\lambda)$ equals the \emph{good filtration dimension} of $L_\lambda$ by Corollary 4.5 in \cite{ParkerGFD}, whence $C_\mathrm{min}(L_\lambda)_{d(\lambda)} \neq 0$ by Lemma 2.15 in \cite{GruberMinimalTilting}.
	Using Lemma 1.12 and Proposition 3.3 in \cite{GruberMinimalTilting} and the linkage principle from Section II.7 in \cite{Jantzen}, it is straightforward to see that for all $i \in \Z$ and $\nu \in X^+$ such that $T_\nu$ is a direct summand of $C_\mathrm{min}(L_\lambda)_i$, we have $\abs{i} \leq d(\lambda) - d(\nu)$.
	In particular, any weight $\mu \in X^+$ such that $T_\mu$ is a direct summand of $C_\mathrm{min}(L_\lambda)_{d(\lambda)}$ satisfies $d(\mu) = 0$ and therefore $(\mu+\rho,\alpha_0^\vee) < p$ (again using the linkage principle), as required.
\end{proof}

\begin{Remark}
	More detailed results about the minimal tilting complexes of simple $G$-modules
	% and of Frobenius twists of $G$-modules
	are proven in Proposition II.2.6
	% and II.6.12 
	in \cite{GruberThesis}; they will be published in a forthcoming article.
\end{Remark}

\begin{Corollary}
	The thick tensor ideal $\mathcal{J}_{2,\leq \abs{\Phi}}$ in $\Rep(G)$ is not of the form $\langle \mathcal{I} \rangle$ for any thick tensor ideal $\mathcal{I}$ in $\Tilt(G)$.
\end{Corollary}
\begin{proof}
	By Lemma \ref{lem:SteinbergFrobeniustwist}, the simple $G$-module $L_{(p^2-p) \cdot \rho}$ belongs to $\mathcal{J}_{2,\leq \abs{\Phi}}$, but it does not belong to $\langle \mathcal{I} \rangle$ for any proper thick tensor ideal $\mathcal{I}$ in $\Tilt(G)$ by Lemma \ref{lem:pregularhighestweight} since the weight $(p^2-p) \cdot \rho$ is $p$-regular.
	(Recall that we assume that $p \geq h = ( \rho , \alpha_0^\vee ) + 1$.)
	As $\mathcal{J}_{2,\leq\abs{\Phi}}$ is a proper thick tensor ideal in $\Rep(G)$, it follows that $\mathcal{J}_{2,\leq\abs{\Phi}}$ is not of the form $\langle \mathcal{I} \rangle$ for any thick tensor ideal $\mathcal{I}$ in $\Tilt(G)$.
\end{proof}

\bibliographystyle{alpha}
\bibliography{tensor}

\end{document}